\documentclass[a4paper,10pt]{amsart}

\usepackage[latin1]{inputenc}
\usepackage{amsmath, amsthm, amssymb}
\usepackage{amscd}
\usepackage[dvips]{graphicx}
\usepackage[all]{xy}
\usepackage{enumerate}

\usepackage{color}

\newtheorem{thm}{Theorem}[section]
\newtheorem{cor}[thm]{Corollary}
\newtheorem{prop}[thm]{Proposition}
\newtheorem{defin}[thm]{Definition}
\newtheorem{lem}[thm]{Lemma}
\newtheorem{claim}[thm]{Claim}
\newtheorem{thmintro}{Theorem}

\newtheorem*{thm*}{Theorem}
\newtheorem*{corspecial}{Corollary \ref{cor:lambda1}}

\theoremstyle{remark}
\newtheorem{rem}[thm]{Remark}
\newtheorem*{remark*}{Remark}

\newcommand{\M}{\widetilde{M}}

\def\R{\mathbb R}
\def\Z{\mathbb Z}
\def\N{\mathbb N}
\def\C{\mathbb C}

\def\S{\mathbb S}
\def\T{\mathbb T}

\def\ada{A \wedge dA^{n-1}}

\def\eps{\varepsilon}

\def\cost{\cos(\theta)}
\def\sint{\sin(\theta)}

\def\cosp{\cos(\phi)}

\def\sinp{\sin(\phi)}
\def\sinpp{\sin^2(\phi)}

\def\cospsi{\cos(\psi)}

\def\sinpsi{\sin (\psi)}

\def\parx{\frac{\partial}{\partial x}}
\def\pary{\frac{\partial}{\partial y}}
\def\partheta{\frac{\partial}{\partial \theta}}

\def\parxx{\frac{\partial^2}{\partial x^2}}
\def\paryy{\frac{\partial^2}{\partial y^2}}
\def\parxy{\frac{\partial^2}{\partial x \partial y}}

\def\parphi{\frac{\partial }{\partial \phi}}
\def\parpsi{\frac{\partial }{\partial \psi}}

\def\parttheta{\frac{\partial^2}{\partial \theta^2}}
\def\parpphi{\frac{\partial^2 }{\partial \phi^2}}

\def\partialxi{\frac{\partial}{\partial x_i}}

\def\ee{\eps^2\sin^2(\phi)}
\def\e{\eps\sin(\phi)}

\def\parfunpsi{\frac{\partial f_1 }{\partial \psi}}
\def\parfdepsi{\frac{\partial f_2 }{\partial \psi}}

\def\xpsi{X_{\psi}}
\def\xphi{X_{\phi}}
\def\xtheta{X_{\theta}}

\def\plm{\tilde{P}_{l}^{m}}
\def\ylm{Y_{l}^{m}}

\def\L{\mathcal{L}}
\def\Le{\mathcal{L}_{\eps}}

\DeclareMathOperator{\Ker}{Ker}
\DeclareMathOperator{\voleucl}{\mathrm{vol}_{\mathrm{Eucl}}}

\title{A natural Finsler--Laplace operator}
\author{Thomas Barthelm\'e}

\begin{document}

\begin{abstract}
 We give a new definition of a Laplace operator for Finsler metric as an average with regard to an angle measure of the second directional derivatives. This definition uses a dynamical approach due to Foulon that does not require the use of connections or local coordinates. We give explicit representations and computations of spectral data for this operator in the case of Katok--Ziller metrics on the sphere and the torus.
\end{abstract}

 \maketitle

\section{Introduction}

The Laplace--Beltrami operator on a Riemannian manifold has long held its place as one of the most important objects in geometric analysis. Among the reasons is that, its spectrum, while being physically motivated, shows an intriguing and intimate connection with the global geometry of the manifold. The Laplacian can be defined in several different ways (see, for example, \cite[Definition 4.7]{GHL}). It can be expressed in coordinate-free ways and admits coordinate representations. In some special cases, its spectrum is effectively computable, but those examples are sparse as they are essentially the round sphere, the Euclidean space, the hyperbolic space and some of their quotients.

The various equivalent definitions of the Laplace-Beltrami operator have motivated extensions to the context of Finsler manifolds \cite{BaoLackey,Centore:FinslerLaplacians,Shen:non-linear_Laplacian} or, in some cases, to their tangent bundles \cite{MR1297076,MR1455407}. The purpose of this paper is to give an extension of the Laplace operator to Finsler manifolds based on the definition of the Laplace-Beltrami operator in terms of second directional derivatives. Our definition produces a symmetric elliptic second-order differential operator. The definition itself is given in coordinate free terms, but we demonstrate that there are adequate coordinate representations and that the spectrum of this Laplacian can be computed effectively.

\subsection{The Finsler--Laplacian}

In the Riemannian context, the Laplace operator can be defined in terms of all second directional derivatives in orthogonal directions. Since there is no suitable notion of orthogonality on Finsler manifolds, the central point in our approach is the introduction of a suitable angle measure $\alpha^F$ that allows us to define a Finsler--Laplace operator as the average of the second directional derivatives: for 
$f \in C^{2}(M)$,
\begin{equation}
  \Delta^F f (x) := c_n \int_{\xi \in T^1_xM} \left. \frac{d^2 f}{dt^2}\left(c_\xi (t) \right)\right|_{t=0}   \alpha^F_x(\xi),
\end{equation}
where $c_n$\ is a normalizing constant (depending only on the dimension of $M$), $T^1_xM$\ is the unit tangent bundle over $x$, $c_\xi$\ is the geodesic leaving $x$\ in the direction $\xi$\ and $\alpha^F$\ is the conditional on the fibers of the canonical volume form on $T^1 M$\ (see Proposition \ref{prop:construction} and Definition \ref{def:delta} for a more precise statement).\\
While constructing the angle $\alpha^F$, we also obtain a natural volume form $\Omega^F$\ on the Finsler manifold. We prove the following:
\pagebreak
\begin{thmintro} \label{thmintro:Laplacian}
 Let $F$\ be a Finsler metric on $M$, then $\Delta^F$\ is a second-order differential operator, furthermore:
\begin{enumerate}
 \item[(i)] $\Delta^F$\ is elliptic, 
 \item[(ii)] $\Delta^F$\ is symmetric, i.e., for any $f,g \in C^{\infty}(M)$, 
 \begin{equation*}
  \int_M f\Delta^F g - g\Delta^F f \; \Omega^F = 0.
 \end{equation*}
 \item[(iii)] Therefore, $\Delta^F$\ is unitarily equivalent to a Schr\"odinger operator.
 \item[(iv)] $\Delta^F$\ coincides with the Laplace--Beltrami operator when $F$\ is Riemannian.
\end{enumerate}
\end{thmintro}

Shen's \cite{Shen:non-linear_Laplacian} extension of the Laplacian is very natural but not linear and hence not comparable to this. Bao-Lackey's \cite{BaoLackey} and Centore's \cite{Centore:FinslerLaplacians} are also elliptic and symmetric but still different from this one (see Remark \ref{rem:comparison_laplacian}). It is somewhat discouraging, as there is no hope to have one canonical Laplacian in Finsler geometry. In fact there are many more ways of generalizing the Laplace operator. Indeed given a Riemannian approximation of a Finsler metric and a volume on the manifold, there is a unique way to associate a Laplace-like operator (see Lemma \ref{lem:existence_unicity}). However, it is common that generalizations of Riemannian objects to Finslerian geometry are far from unique; see, for instance, discussions about  volumes (see \cite{BuragoBuragoIvanov}) or the different connections and the notions of curvature (see \cite{Egloff:thesis,BaoChernShen}). Hence, the goal was to find an extension such that its definition seems ``natural'' and that enjoys links with the geometry.

Our approach follows a dynamical point of view introduced by P. Foulon \cite{Fou:EquaDiff} that does not require local coordinate computations or the Cartan or Chern connections (see \cite{Crampon1,Egloff:DynamicsUFM,Egloff:UFHM,Fou:EESLSPC,Fou:LSFSNC} for some results obtained via this approach). A consequence is that, for any contact form on the homogeneous bundle $HM$ (see section \ref{notations}) of a manifold, we can define a Laplace operator associated with its Reeb field. We will not emphasize this more general setting, but it should be clear that every result that still makes sense in this greater generality stays true. Another consequence is that our operator is essentially linked to the geodesic flow, hence one could hope that this link will reappear in the spectral data.\\

\subsection{Spectrum}

We show that there is a natural energy functional $E$, linked to $\Delta^F$, such that harmonic functions are obtained as minima of that functional (Theorem \ref{th:min_energy}). Furthermore, as is expected from a Laplacian, general theory shows that when $M$ is compact, $-\Delta^F$ admits an unbounded, positive, discrete spectrum (Theorem \ref{thm:spectre_discret}) and we can obtain it from the energy via the min-max principle (Theorem \ref{thm:min_max}). 
In Riemannian geometry, it is known that the Laplace--Beltrami operator is a conformal invariant only in dimension two. Using the energy, we can push the proof to our context:\emph{ If $(\Sigma,F)$\ is a Finsler surface, $f \colon \Sigma \xrightarrow{C^{\infty}} \R$\ and $F_f = e^f F$, then, $ \Delta^{F_f} = e^{-2f} \Delta^F$}.

\subsection{Coordinate representation and computation of spectrum}

Our goal is to introduce this new operator, state its basic properties and study some explicit examples where spectral data can be computed. Indeed, we feel that the computability of examples is an important feature of this operator.\\
In the Riemannian case, the spectrum is known only for constant-curvature spaces. So it is natural to study Finsler metrics with constant flag curvature (cf. \cite{BaoChernShen,Egloff:thesis}).\\
 If the flag curvature is negative, then a theorem of Akbar-Zadeh \cite{AkZ} implies that, if the manifold is closed, then the Finsler structure is in fact Riemannian.\\
 In the same article, Akbar-Zadeh also showed that a simply connected compact manifold endowed with a metric of positive constant flag curvature is a sphere. Bryant \cite{Bryant:FS2S,Bryant:PFF2S} constructed such examples. Previously, Katok \cite{Katok:KZ_metric} had constructed a family of one-parameter deformations of the standard metric on $S^2$\ in order to obtain examples of metrics with only a finite number of closed geodesics. This example was later generalized and studied by Ziller \cite{Ziller:GKE}. We now know that these Katok--Ziller metrics on the sphere have constant flag curvature \cite{Rademacher:Sphere_theorem}. Another asset of these metrics is that they admit adequate explicit formulas (see \cite{Rademacher:Sphere_theorem} or Proposition \ref{prop:KZ_explicit}) making them somewhat easier to study. Therefore, we choose to study the spectrum of our operator for these Katok--Ziller metrics in the case of the $2$-sphere.\\
 For instance, computation of spectral data in the $\S^2$ case gives the following:
\begin{thmintro} \label{thmintro:lambda1}
 For a family of Katok--Ziller metrics $F_{\eps}$ on the $2$-sphere, if $\lambda_1(\eps)$ is the smallest non-zero eigenvalue of $-\Delta^{F_{\eps}}$, then:
\begin{equation}
 \lambda_1(\eps) = 2 -2 \eps^2 = \frac{8 \pi}{\rm{vol}_{\Omega}\left(\S^2\right)}.
\end{equation}
\end{thmintro}
Note that this result exhibits a family of Finslerian metrics realizing what is known to be the maximum for the first eigenvalue of the Laplace--Beltrami operator on $\S^2$.\\
 Finally as Katok--Ziller metrics also exist on the $2$-torus, we study them. Note that the flat case will not lead to new operators: The Finsler--Laplace operator in that case is the same as the Laplace--Beltrami operator associated with the symbol metric (see Remark \ref{rem:utilisation_unique1}). In fact, this clearly stays true for any locally Minkowski structure on a torus (see \cite{moi:these}). It is nonetheless interesting to do the computations as it gives some insight, shows some limits of what can be expected from this operator and proves again that computations are feasible.\\

\begin{remark*}
 Some Finsler geometers like to consider only reversible metrics (see Definition \ref{def:finsler_metric} below) but this entails a severe loss of generality. For instance, Bryant \cite{Bryant:GR} showed that the only reversible metrics on $\S^2$\ of constant positive curvature are Riemannian.
\end{remark*}

\subsection{Laplacian and geometry at infinity}
This article concentrates on providing a foundation for the study of this Finsler-Laplacian. To indicate that there are deep links between the dynamics of the geodesic flow and this operator, we annouce here adaptations of two classical Riemannian results (due to Sullivan \cite{Sul:Dirichlet_at_infinity} and Anderson and Schoen \cite{AndersonSchoen} for the first and Ledrappier \cite{Ledrappier:Ergodic_properties} for the second, the Finsler versions can be found in \cite{moi:these}):
\begin{thm*}
 Let $F$ be a \emph{reversible} Finsler metric of negative flag curvature on a closed manifold $M$, $(\widetilde{M},\widetilde{F})$ the lifted structure on the universal cover of $M$ and $\widetilde{M}(\infty)$ its visual boundary. Then, for any function $f \in C^0(\widetilde{M}(\infty) ) $, there exists a unique function $u \in C(\M\cup \M(\infty))$ such that 
\begin{equation*}
 \left\{ \begin{aligned}
          \Delta^{\widetilde{F}}u &= 0 \quad \text{on } \M \\
	  u &= f \quad \text{on } \M(\infty)
         \end{aligned} \right.
\end{equation*}
Furthermore, for any $x\in \M$, there exists a unique measure $\mu_x$, called the \emph{harmonic measure} for $\Delta^{\widetilde{F}}$ such that:
\begin{equation*}
 u(x):= \int_{\xi \in \M(\infty)} f(\xi) d\mu_x(\xi).
\end{equation*}
\end{thm*}

\begin{thm*}
Let $(M,F)$ and $\mu_x$ be as above, we have the following properties:
\begin{itemize}
 \item[(i)] The harmonic measure class $\lbrace \mu_x \rbrace$ is ergodic for the action of $\pi_1(M)$ on $\M(\infty)$.\\
 \item[(ii)] For any $x\in \M$, the product measure $\mu_x \otimes \mu_x$ is ergodic for the action of $\pi_1(M)$ on $\partial^2\M := \M(\infty) \times \M(\infty) \smallsetminus \text{diag}$.\\
 \item[(iii)] There exists a unique geodesic flow invariant measure $\mu$ on $HM$ such that the family of spherical harmonics $\nu_x$ is a family of transverse measures for $\mu$. Moreover $\mu$ is ergodic for the geodesic flow.
\end{itemize}
\end{thm*}

\subsection{Structure of this paper}

In section \ref{sec:angle_form} we introduce our notion of (solid) angle $\alpha^F$\ in Finsler geometry together with the volume form $\Omega^F$. The volume form turns out to be the Holmes-Thompson volume \cite{HolmesThompson} associated with $F$, but it seems that the angle form has not been used or studied, maybe even introduced, previously. We also state some of the properties of the angle.\\
Section \ref{sec:laplacian} is devoted to the definition of the Finsler--Laplace operator and the proof of Theorem A.\\
In section \ref{sec:energy_and_spectrum}, we define an energy associated with our Finsler--Laplace operator and we show that the harmonic functions are its minima. We recall that this operator, as in the Riemannian case, admits a discrete spectrum when the manifold is compact. We also show that the Finsler--Laplace operator on surfaces is a conformal invariant.\\
The last section gives explicit representations of our operator and spectrum information for Katok--Ziller metrics on the sphere and the torus.

\subsection{Notations}
\label{notations}

Throughout this text, $M$\ stands for a smooth manifold of dimension $n$\ and $F$\ a Finsler structure on it.
\begin{defin} \label{def:finsler_metric}
A smooth Finsler metric on $M$\ is a continuous function $F \colon TM \rightarrow \R^+$ that is:

\begin{enumerate}
  \item $C^{\infty}$\ except on the zero section,
  \item positively homogeneous, i.e., $F(x,\lambda v)=\lambda F(x,v)$\ for any $\lambda>0$,
  \item positive-definite, i.e., $F(x,v)\geq0$\ with equality iff $v=0$,
  \item strongly convex, i.e., $ \left(\dfrac{\partial^2 F^2}{\partial v_i \partial v_j}\right)_{i,j}$ is positive-definite.
 \end{enumerate}
\end{defin}
It is said to be \emph{reversible} if $F(x, -v ) = F(x,v)$\ for any $(x,v)\in TM$.\\

We write $HM$\ for the homogenized bundle, i.e., $ HM := \left(TM \smallsetminus \{\text{zero section}\} \right) / \R^+$. We have two natural projections $r \colon TM \rightarrow HM$\ and $\pi \colon HM \rightarrow M$\ as well as an associated vertical bundle $VHM = \Ker d\pi$, where $d\pi \colon THM \rightarrow TM$\ is the derivative of $\pi$.\\

The \emph{Hilbert form} $A$\ associated with $F$\ is defined as the projection on the homogenized bundle of the vertical derivative of $F$:
$$
r^{\ast}A = d_v F,
$$
where $\displaystyle d_v F_z (\xi) := \lim_{h\rightarrow 0 } \frac{F\left(z+h Tp(\xi) \right)}{h}$\ for $z\in TM$\ and $\xi \in T_z TM$\ (called the vertical derivative).
In local coordinates $\left(x^i, v^j\right)$, the vertical derivative reads:
$$
d_vF = \frac{\partial F}{ \partial v^i } dx^i.
$$

 Under our assumptions on $F$, $A$\ is a contact form, with associated Reeb field $X$\ being the generator of the geodesic flow (see \cite{Fou:EquaDiff}). By definition, we have:
\begin{equation}
\label{eq:Reeb_field}
\left\{ 
\begin{aligned}
  A(X) &= 1 \\
 i_X dA &= 0  
 \end{aligned}
\right.
\end{equation}
This implies that the volume is invariant by the flow, i.e.,
\begin{equation}
\label{eq:lxada}
 L_X\left(\ada\right) = 0.
\end{equation}

\section{Angle form} \label{sec:angle_form}

This section is devoted to the construction of an angle form, i.e., an $(n-1)$-form on $HM$\ which is never zero on $VHM$, and to the study of some of its properties.

\subsection{Construction}
\label{par:construction}
We split the natural volume form $\ada$\ on $HM$\ into a vertical part and a part coming from the base manifold $M$.\\

\begin{prop}
\label{prop:construction}
 There exists a unique volume form $\Omega^F$\ on $M$\ and an $(n-1)$-form $\alpha^F$\ on $HM$ that is nowhere zero on $VHM$ and such that:
\begin{equation}
\label{eq:alpha_wedge_omega}
  \alpha^{F} \wedge \pi^{\ast}\Omega^F =  A\wedge dA^{n-1}, 
\end{equation}
and, for all $x\in M$, 
\begin{equation}
\label{eq:longueur_fibre}
 \int_{H_xM} \alpha^F =  \voleucl(\S^{n-1})
\end{equation}

\end{prop}

\begin{rem}
 We do not claim that the angle form $\alpha^F$\ is unique (we can add any $(n-1)$-form that is null on $VHM$\ and still satisfy the above conditions). However for any open set $U$\ of $H_xM$, $\int_U \alpha^F$\ is well defined and does not depend on the choice of a such $\alpha^F$. Hence, we do have what we want: a notion of solid angle.
\end{rem}

\begin{proof}
Let $\omega$\ be a volume form on $M$. There exists an $(n-1)$-form $\alpha^{\omega}$\ on $HM$\ such that $\alpha^{\omega} \wedge \pi^{\ast}\omega =  A\wedge dA^{n-1}$. This equation characterizes $\alpha^{\omega}$\ up to a form that is null on $VHM$. Indeed, for linearly independent vertical vector fields $Y_1, \dots, Y_{n-1} $, we have:
\begin{equation*}
 \alpha^{\omega} \left(Y_1, \dots, Y_{n-1} \right) = \frac{A\wedge dA^{n-1} \left( Y_1, \dots, Y_{n-1}, X, \left[X, Y_1 \right], \dots, \left[X, Y_{n-1}\right] \right) } {\pi^{\ast}\omega\left( X, \left[X, Y_1 \right], \dots, \left[X, Y_{n-1}\right] \right) }.
\end{equation*}

As $\alpha^{\omega}$\ is uniquely determined on $VHM$, it makes sense to integrate it over the fibers. For any $x\in M$, set 
\begin{equation*}
 l^{\omega}(x):= \int_{H_x M} \alpha^{\omega}.
\end{equation*}
$l^{\omega}$\ might not be constant, but we can choose $\omega$\ such that it is. Let 
\begin{equation}
\label{eq:definition_Omega_F}
\Omega^F := \frac{l^{\omega}(x)}{\voleucl\left(\S^{n-1}\right)} \omega,
\end{equation}
and $\alpha^F$\ given by \eqref{eq:alpha_wedge_omega}. Then 
\begin{equation*}
 \alpha^{\omega} \wedge \pi^{\ast}\omega  =   \alpha^F \wedge \pi^{\ast} \Omega^F = \frac{l^{\omega}}{\voleucl\left(\S^{n-1}\right)} \alpha^F \wedge \pi^{\ast} \omega .
\end{equation*}
Therefore, for any $Y_1, \dots, Y_{n-1} $\ in $VHM$,
\begin{equation*}
 \alpha^F (Y_1, \dots, Y_{n-1} ) = \frac{\voleucl\left(\S^{n-1}\right)}{l^{\omega}} \alpha^{\omega}( Y_1, \dots, Y_{n-1} ),
\end{equation*}
which yields, for any $x \in M$,
\begin{equation*}
 \int_{H_x M}\alpha^F = \voleucl\left(\S^{n-1}\right).
\end{equation*}

The uniqueness of $\Omega^F$\ is straightforward.
\end{proof}

Note that Finsler geometry can also be studied via its Hamiltonian/symplectic side, which often yields some very interesting result, we could have presented the above construction in that setting (we do it in \cite{moi:these}), however, we felt that, for the material presented in this article, the Hamiltonian setting was not better. The only exception being the following (see the proof in \cite{moi:these}):
\begin{lem}
 $\dfrac{\Omega^F}{(n-1)!}$ is the Holmes-Thompson volume associated with $F$.
\end{lem}
In the sequel, we will often write $\alpha$\ and $\Omega$ for the angle and volume form when the Finsler metric we use is clear.

\begin{rem}
 On Finsler surfaces, the angle $\alpha$ generates rotations: Indeed, there exists a unique vertical vector field $Y$ such that $\alpha(Y)=1$, so if $R^t$\ is the one-parameter group generated by $Y$, then $\forall (x,v) \in HM$, $t\in \R$,
\begin{equation*}
\label{eq:rotation}
\left\{
\begin{aligned}
 \pi\left(R^t(x,v)\right) &= x \\
 R^{2\pi}(x,v) &= (x,v)
\end{aligned}
\right.
\end{equation*}
And if the Finsler metric is \emph{reversible}, we also have (see \cite{moi:these} for the proof):
\begin{equation*}
 R^{\pi}\left(x,v\right) = \left(x,-v\right).
\end{equation*}
\end{rem}

\subsection{Behavior under conformal change}

\begin{prop}
Let $(M,F)$\ be a Finsler manifold, $f \colon M \xrightarrow{C^{\infty}} \R$, $F_f= e^{f} F$, $\alpha_f$\ and $\Omega_f$\ the angle and volume form of $F_f$. Then $\alpha_f = \alpha$\ and $\Omega_f = e^{nf} \Omega$.
\end{prop}

 \begin{proof}
  Using the definition of the Hilbert form, we immediately have $A_f = e^{f} A$, so 
  $$
  A_f \wedge dA_f^{n-1} = e^{nf} A\wedge dA^{n-1}.
  $$
  Let $\omega$\ be a volume form on $M$. Let $\alpha_F^{\omega}$\ and $\alpha_{F_f}^{\omega}$\ be the two $(n-1)$-forms defined by $\alpha_F^{\omega}\wedge \pi^{\ast} \omega =  A\wedge dA^{n-1}$\ and $\alpha_{F_f}^{\omega}\wedge \pi^{\ast} \omega  =  A_f \wedge dA_f^{n-1}$. We have, 
$$
\alpha_{F_f}^{\omega}\wedge \pi^{\ast} \omega = e^{nf} \alpha_F^{\omega}\wedge \pi^{\ast} \omega.
$$
From there we get that, for any $Y_1, \dots , Y_{n-1} \in VHM$:
\begin{equation}
\alpha_{F_f}^{\omega}\left(Y_1, \dots , Y_{n-1} \right) = e^{nf} \alpha_{F}^{\omega}\left(Y_1, \dots , Y_{n-1} \right).
\end{equation}
We deduce that for any $x\in M$,
\begin{equation*}
 \int_{H_xM} \hspace{-2mm} \alpha_{F_f}^{\omega} = e^{nf(x)} \int_{H_xM} \hspace{-2mm} \alpha_{F}^{\omega}, 
\end{equation*}
The two volume forms $\Omega$\ and $\Omega_f$\ on $M$\ associated with $F$\ and $F_f$\ are given by (see equation \eqref{eq:definition_Omega_F}):

\begin{equation*} 
\Omega_f = \frac{\int_{\scriptscriptstyle{H_xM}} \hspace{-1mm} \alpha_{F_f}^{\omega}}{c_n} \omega,\; \text{and} \; \, \Omega = \frac{\int_{\scriptscriptstyle{H_xM}} \hspace{-1mm} \alpha_{F}^{\omega}}{c_n} \omega,
\end{equation*}
which yields
\begin{equation}
\label{eq:omega1}
 \Omega_f = e^{nf} \Omega.
\end{equation}
Using the definition of $\alpha_f$\ and equation (\ref{eq:omega1}), we obtain:
\begin{equation*}
e^{nf} \alpha \wedge \pi^{\ast}\Omega  = e^{nf} \alpha_f \wedge \pi^{\ast} \Omega. \label{eq:alpha_et_alpha1}
\end{equation*}
This yields that, for any $Y_1, \dots , Y_{n-1} \in VHM$, we have
\begin{equation*}
\alpha\left(Y_1, \dots , Y_{n-1}\right) = \alpha_f\left(Y_1, \dots , Y_{n-1}\right). \qedhere
\end{equation*}

 \end{proof}

\section{Finsler--Laplace--Beltrami operator} \label{sec:laplacian}

We start this section with the definition of our Finsler--Laplace operator. The reader can check that it is the same as in the introduction. The aim of the rest of the section is to prove Theorem \ref{thmintro:Laplacian}.

\begin{defin}
\label{def:delta}
 We define $\Delta^F$\ by 
 $$
 \Delta^F f (x) = \frac{n}{\voleucl \left(\mathbb{S}^{n-1}\right) }\int_{H_xM} L_X ^2 (\pi^{\ast} f ) \alpha^F,
 $$
 for every $x\in M$\ and every $f \colon M \rightarrow \R$\ (or $\C$) such that the integral exists.
\end{defin}

As we will see in the next section, the constant $\frac{n}{\voleucl \left(\mathbb{S}^{n-1}\right) }$\ is chosen so that $\Delta^F$\ is the Laplace--Beltrami operator when $F$\ is Riemannian.
\begin{rem}
 To define this operator we just needed the contact form $A$ on $HM$, not the full Finsler metric and the results in the sequel of this article would remain true. It is also clear from the definition that $\Delta^F$\ is a linear differential operator of order two.
\end{rem}

\subsection{The Riemannian case}

We start with the proof of Theorem \ref{thmintro:Laplacian}(iv).
\begin{prop}
 Let $g$\ be a Riemannian metric on $M$, $F = \sqrt{g}$, $\Delta^F$\ the Finsler--Laplace operator and $\Delta^g$\ the usual Laplace--Beltrami operator. Then,
\begin{equation*}
 \Delta^F = \Delta^g.
\end{equation*}

\end{prop}

\begin{proof}
We compute both operators in normal coordinates for $g$.\\
Let $p\in M$\ and $x_1, \dots, x_n$\ the normal coordinates around it. Denote by $v_1, \dots , v_n$\ their canonical lift to $T_x M$.
For $f \colon M \rightarrow \R$, the Laplace--Beltrami operator gives  $ \Delta^{g}f(p)  = \displaystyle{ \sum_i } \frac{\partial^2 f (p)}{\partial x_i^2} $.\\
The first step to compute the Finsler--Laplace operator is to compute the Hilbert form $A$\ and the geodesic flow $X$. In order to write $A$, we identify $HM$\ with $T^1M$\ and coordinates on $H_pM$\ are then given by the $v_i$'s with the condition $\sqrt{\sum (v_i)^2 }=1$. The vertical derivative of $F$\ at $p$\ is $d_v F_p = \tfrac{v_i}{\sqrt{\sum (v_i)^2 }} dx_i$. So  $ A_p = v_i\;dx^i$\ and $dA_p = dv_i \wedge dx^i $. Hence $X(p, \cdot ) = v_i \partialxi$. Indeed, we just need to check that $A_p\left(X_p\right) = 1$\ and $\left(i_X dA \right)_p = 0 $: both equalities follow from $\sum (v_i)^2 =1$.\\
Let $f \colon M \rightarrow \R$, then
\begin{equation*}
 L_X^2\left(\pi^{\ast}f \right) (p,v) = v_i v_j \frac{\partial^2 f}{\partial x_i\partial x_j} (p,v).
\end{equation*}
The Finsler--Laplace operator is:
\begin{equation*}
 \Delta^F f(p) = \frac{ n}{\voleucl \S^{n-1}} \int_{H_pM} v_i v_j \; \alpha \; \frac{\partial^2 f}{\partial x_i\partial x_j}(p),
\end{equation*}
And the proof follows from the next two claims. \qedhere

\end{proof}

\begin{claim}
 For all $i \neq j$, 
$$
\int_{H_pM} v_i v_j \, \alpha =0
$$
\end{claim}

\begin{proof}
 $H_pM$\ is parametrized by $ H_pM = \lbrace (v_1, \dots , v_n ) \mid v_i \in [-1, 1] \rbrace.$\\
 A parity argument then yields the desired result. \qedhere

\end{proof}

\begin{claim}
 For any $1\leq i \leq n $,
$$
\int_{H_pM} v_i^2 \, \alpha = \frac{\voleucl \S^{n-1}}{ n}
$$

\end{claim}

\begin{proof}
 As the $v_i$'s are symmetric by construction, we have that for any $i\neq j$,
$$
\int_{H_pM} v_i^2 \, \alpha = \int_{H_pM} v_j^2 \, \alpha.
$$
So,
\begin{equation*}
 n \int_{H_pM} \hspace{-1.8mm} v_i^2 \, \alpha =  \sum_j \int_{H_pM} \hspace{-1.8mm} v_j^2 \, \alpha =   \int_{H_pM} \sum_j v_j^2 \, \alpha  =   \int_{H_pM} \hspace{-2.3mm}1 \, \alpha  = \voleucl \S^{n-1}. \qedhere
\end{equation*}

\end{proof}

\subsection{Ellipticity}
We give here the proof of Theorem \ref{thmintro:Laplacian} (i) and an expression for the symbol.
\begin{prop}
 $\Delta^F \colon  C^{\infty}(M) \rightarrow  C^{\infty}(M) $\ is an elliptic operator. The symbol $\sigma^F$ is given by
\begin{equation*}
 \sigma^F_x(\xi_1,\xi_2) = \frac{n}{\voleucl \left(\mathbb{S}^{n-1}\right) } \int_{H_xM} L_X(\pi^{\ast} \varphi_1) L_X(\pi^{\ast}\varphi_2)\, \alpha^F
\end{equation*}
for $\xi_1,\xi_2 \in T^{\ast}_x M$, where $\varphi_i \in C^{\infty}(M)$ such that $\varphi_i(x)=0$ and $\left.d\varphi_i\right._x = \xi_i$.
\end{prop}

\begin{rem}
 If we identify the unit tangent bundle $T^1 M$ with the homogenized tangent bundle $HM$ and write again $\alpha^F$ for the angle form on $T^1M$, then the symbol is given by
\begin{equation*}
 \sigma^F_x(\xi_1,\xi_2) = \frac{n}{\voleucl \left(\mathbb{S}^{n-1}\right) } \int_{v\in T^1_xM} \xi_1(v) \xi_2(v) \, \alpha^F(v)
\end{equation*}
for $\xi_1,\xi_2 \in T^{\ast}_x M$.\\
The symbol of an elliptic second-order differential operator is a non-degenerate symmetric $2$-tensor on the cotangent bundle, and therefore defines a Riemannian metric on $M$. This gives one more way to obtain a Riemannian metric from a Finsler one. Let $\Delta^{\sigma}$\ be the Laplace--Beltrami operator associated with the symbol metric, then $\Delta^F -\Delta^{\sigma}$\ is a differential operator of first order, so is given by a vector field $Z$\ on $M$. The Finsler--Laplace operator therefore is a Laplace--Beltrami operator together with some ``drift'' given by $Z$. We will see that our operator is in fact characterized by its symbol and the symmetry condition.
\end{rem}

\begin{proof}
 To show that $\Delta$\ is elliptic at $p\in M$, it suffices to show that for each $\varphi \colon M  \rightarrow \R $\ such that $\varphi(p) = 0$\ and $d\varphi|_p$\ is non-null, and for $u \colon M \rightarrow \R^+$ we have $\Delta^F(\varphi^2 u) (p) > 0$ unless $u(p)=0$.\\ 
 We first compute $L^{2}_X \left(\pi^{\ast}\varphi^2 u\right)$:
 \begin{align*}
  L^{2}_X \left(\pi^{\ast}\varphi^2 u\right) &= L_X \left( 2 \varphi u L_X\left( \pi^{\ast}\varphi\right) + \varphi^2 L_X \left(\pi^{\ast}u\right) \right), \\
  &=  2 u \left(L_X\left( \pi^{\ast}\varphi\right)\right)^2  + 2 \varphi u L^{2}_X \left(\pi^{\ast}\varphi \right) \\
   & \quad   +  4 \varphi L_X \left(\pi^{\ast}\varphi\right) L_X \left(\pi^{\ast}u\right) + 2\varphi^2 L^{2}_X\left( \pi^{\ast}u\right) .      
 \end{align*}
Evaluating in $\left(p,\xi\right) \in HM$, we obtain, 
\begin{equation*}
 L^{2}_X \left(\pi^{\ast}\varphi^2 u\right) \left( p,\xi\right) = 2 u(p) \left(L_X \pi^{\ast}\varphi\right)^2 \left( p,\xi\right).
\end{equation*}
 Therefore,
 \begin{align*}
 \Delta^F(\varphi^2 u) (p) &= \frac{n}{\voleucl \left(\S^{n-1}\right) } \int_{H_p M} 2 u(p) \left(L_X \pi^{\ast}\varphi\right)^2 \alpha ,\\
  &=  \frac{2 u(p)n}{\voleucl \left(\mathbb{S}^{n-1}\right) } \int_{H_p M} \left(L_X \pi^{\ast}\varphi\right)^2 \alpha \; > 0. \qedhere
 \end{align*}

\end{proof}

\subsection{Symmetry}

We have an hermitian product defined on the space of $C^{\infty}$\ complex functions on $M$\ by 
\begin{equation*}
\langle f,g \rangle = \int_M f(x)\overline{g(x)} \Omega^F.
\end{equation*}
We have (Theorem \ref{thmintro:Laplacian} (ii)):
\begin{prop}
\label{prop:symmetry}
Let $M$\ be a closed manifold, then $\Delta^F$\ is symmetric for $\langle \cdot , \cdot \rangle $\ on $C^{\infty}(M)$, i.e., for any $f,g\in C^{\infty}(M)$, we have:
 \begin{equation*}
 \langle \Delta^F f , g \rangle =  \langle f,\Delta^F g \rangle.
 \end{equation*}
\end{prop}

\begin{rem}
 The proof of this result is remarkably simple due to our choice of angle form and volume. Indeed, as $\alpha\wedge \pi^{\ast} \Omega$ is the canonical volume on $HM$, it is invariant by the geodesic flow (i.e., $L_X(\alpha\wedge \pi^{\ast} \Omega) =0$) which is the key to the computation.
\end{rem}

To prove Proposition \ref{prop:symmetry}, we shall use a Fubini-like result:

\begin{lem}
 Let $f:HM \rightarrow \C$\ be a continuous function on $HM$. We have,
 \begin{equation}
 \int_M \left( \int_{H_xM} f(x,\cdot ) \, \alpha \right) \Omega = \int_{HM} f \; \alpha \wedge \pi^{\ast}\Omega.
 \end{equation}  
\end{lem}

We can now proceed with the

 \begin{proof}[Proof of Proposition \ref{prop:symmetry}]
  Let $f,g: M \xrightarrow{C^{\infty}} \C$\ and write $c_n := \frac{n}{\voleucl \left(\S^{n-1}\right) }$.
  \begin{align*}
   \langle \Delta^F f,g \rangle &= \int_M \overline{g} \Delta^F f  \;  \Omega \\
                  &= c_n \int_M \overline g \left( \int_{H_xM} L_X ^2 (\pi^{\ast} f ) \alpha \right) \Omega  \\
		  &= c_n \int_M \left( \int_{H_xM} \overline{\pi^{\ast}g} L_X ^2 (\pi^{\ast} f ) \alpha\right) \Omega  \\
		  &= c_n \int_{HM} \overline{\pi^{\ast}g} L_X ^2 (\pi^{\ast} f ) \; \alpha\wedge\pi^{\ast}\Omega,
  \end{align*}
  where the last equality follows from the preceding lemma. As $\alpha\wedge\pi^{\ast}\Omega = A\wedge dA^{n-1}$\ we can write 
  $$
  \langle \Delta^F f,g \rangle = c_n \int_{HM} \overline{\pi^{\ast}g} L_X ^2 (\pi^{\ast} f ) \; A\wedge dA^{n-1}.
  $$
Now, 
  \begin{multline*} 
  \label{eq:1}  
   L_X \left( \overline{\pi^{\ast}g} L_X (\pi^{\ast} f ) A\wedge dA^{n-1} \right)  =  \overline{\pi^{\ast}g}  L_X^2 (\pi^{\ast} f ) A\wedge dA^{n-1}  \\
    + L_X ( \overline{\pi^{\ast}g} ) L_X (\pi^{\ast} f ) A\wedge dA^{n-1}  + \overline{\pi^{\ast}g} L_X (\pi^{\ast} f ) L_X (\ada).
   \end{multline*}
The last part of the above equation vanishes because of \eqref{eq:lxada}. We also have:
\begin{equation*} 
   L_X \left( \overline{\pi^{\ast}g} L_X (\pi^{\ast} f ) A\wedge dA^{n-1} \right)  = d\left( i_X \overline{\pi^{\ast}g} L_X (\pi^{\ast} f ) A\wedge dA^{n-1} \right).
\end{equation*}
Hence 
\begin{multline*}
 \langle \Delta^F f,g \rangle = \frac{n}{\voleucl \left(\S^{n-1}\right) }\Biggl[ \int_{HM} d\left( i_X \overline{\pi^{\ast}g} L_X (\pi^{\ast} f ) A\wedge dA^{n-1} \right)  \\
     -  L_X ( \overline{\pi^{\ast}g} ) L_X (\pi^{\ast} f ) A\wedge dA^{n-1} \Biggr].
\end{multline*}
As $M$\ is closed, $HM$\ is closed and applying Stokes Theorem gives \eqref{eq:green_formula}, thus proving the claim. 

 \end{proof}

In the proof we obtained a Finsler version of Green's formulas:
 \begin{prop}
 \label{prop:green_formula}
\begin{enumerate}
 \item For any $f,g \in C^{\infty}(M)$, we have:
  \begin{equation} \label{eq:green_formula}
  \langle\Delta^F f,g \rangle = \frac{-n}{\voleucl \left(\S^{n-1}\right) } \int_{HM} L_X ( \overline{\pi^{\ast}g} ) L_X (\pi^{\ast} f ) A\wedge dA^{n-1}.
  \end{equation}
 \item Let $U$ be a submanifold of $M$\ of the same dimension and with boundaries. Then for any $f \in C^{\infty}(U)$, we have:
 \begin{equation}
 \int_U \Delta^F f \; \Omega^F =\frac{n}{\voleucl \left(\S^{n-1}\right) }\int_{\partial HU }  L_X (\pi^{\ast} f ) dA^{n-1} 
 \end{equation}
\end{enumerate}
 \end{prop}

\subsection{A characterization of $\Delta^F$}

The following results were explained to me by Yves Colin de Verdi\`ere to whom I am very grateful.
Up until now, we have associated an elliptic symmetric operator, a volume form and a Riemannian metric, via the (dual of the) symbol, to a Finsler structure on a manifold. But in fact, the latter two suffice to define our Finsler-Laplace operator.
\begin{lem}\label{lem:existence_unicity}
 Let $(M,g)$\ be a closed Riemannian manifold and $\omega$\ a volume form on $M$. There exists a unique second-order differential operator $\Delta_{g,\omega}$\ on $M$ with real coefficients such that its symbol is the dual metric $g^{\star}$, it is symmetric with respect to $\omega$ and zero on constants.\\
If $a\in C^{\infty}(M)$ is such that $\omega = a^2 v_g$, where $v_g$ is the Riemannian volume, then for $\varphi \in C^{\infty}(M)$:
\begin{equation*}
 \Delta_{g,\omega} \varphi = \Delta^{\textrm{LB}}_g \varphi - \frac{1}{a^2} \langle \nabla \varphi , \nabla a^2 \rangle.
\end{equation*}

\end{lem}

\begin{rem}
\begin{itemize}
 \item Up until now we only considered operators on  $C^{\infty}(M)$. However, for spectral theory purpose, it is convenient to consider them as \emph{unbounded} operator on $L^2(M,\omega)$ (see \cite{ReedSimon:I}). For such operators, there is a difference between symmetric and self-adjoint. However, the operators considered in this article admits an extension, called the Friedrich extension, that is self-adjoint (see \cite{Kato:Perturbation_theory,ReedSimon:II}). Using this extension, the previous lemma stays true replacing symmetric by self-adjoint.
 \item This lemma shows that there must be many Finsler metrics giving the same Laplacian. It would be interesting to see whether all the couples $(g,\omega)$ can arise from a Finsler metric via our construction, or if there is another obstruction that limits the scope of our possible operators.
 \item The operators $\Delta_{g,\omega}$ are called \emph{weighted Laplace operator} and were originally introduced by Chavel and Feldman \cite{ChavelFeldman:Isoperimetric_constants} and Davies \cite{Davies:Heat_kernel_bounds}, some further work on them and references can be found in \cite{Grigoryan:heat_kernels_on_weighted_manifolds}.
\end{itemize}

\end{rem}

\begin{proof}
It is evident from the definition of $\Delta_{g,\omega}$ that its symbol is $g^{\ast}$ and that for $\varphi, \psi \in \C^{\infty}(M)$, 
\begin{equation*}
\int_M \psi \Delta_{g,\omega}\varphi \; \omega = \int_{M} g^{\ast}\left(d\varphi, d\psi \right) \, \omega. 
\end{equation*}

 Let us now prove the uniqueness. Let $\Delta_1$\ and $\Delta_2$ be two second-order differential operators such that they are null on constants and have the same symbol. This implies that there exists a smooth vector field $Z$ on $M$ such that $\Delta_1 - \Delta_2 = L_Z$. Now, let us suppose that both operators are symmetric with respect to $\omega$.\\
 We have, $\int_M \varphi L_Z \psi -\psi L_Z \varphi \, \omega =0$ for any $\varphi, \psi \in \C^{\infty}(M)$. Taking $\psi = 1$ yields $\int_M L_Z \varphi \, \omega =0$. Now, it is easy to construct a function $\varphi \in \C^{\infty}(M)$ such that $L_Z \varphi >0$ in any open set that does not contain a singular point of $Z$. By continuity, $Z$ must be null.
\end{proof}

An important consequence of this lemma is that any symmetric, elliptic linear second order operator is unitarily equivalent to a Schr\"odinger operator, hence proving Theorem \ref{thmintro:Laplacian} (iii).
\begin{prop}
 Let $\Delta_{g,\omega}$, $v_g$ and $a$ be as above. Define $U \colon L^2\left(M, \omega\right) \rightarrow L^2\left(M, v_g \right)$ by $Uf = af$. Then $U \Delta_{g,\omega} U^{-1} = \Delta^{LB}_g + V $ is a Schr\"odinger operator with potential $V = a\Delta_{g,\omega} a^{-1}$.
\end{prop}

\begin{rem}
 This fact shows that the spectral theory of our operator restricts to the theory for Schr\"odinger operators such that the infimum of the spectrum is zero.
\end{rem}

\begin{proof}
 It suffices to show that $U \Delta_{g,\omega} U^{-1} - V$ is symmetric with respect to $\omega$ and has $g^{\ast}$ for symbol, because then Lemma \ref{lem:existence_unicity} proves the claim. The symmetry property is obvious by construction. Let $x\in M$ and $\varphi \in L^2\left(M, v_g \right)$ be such that $\varphi (x) = 0$ and $d\varphi_{x} \neq 0$. We have
\begin{equation}
 \left(U \Delta_{g,\omega} U^{-1} - V\right) \varphi^2 (x) = a\Delta_{g,\omega}(\varphi^2 a^{-1})(x) = \Delta_{g,\omega}(\varphi^2)(x).
\end{equation}
Therefore the symbol of $\left(U \Delta_{g,\omega} U^{-1} - V\right)$ is the same as that of $\Delta_{g,\omega}$.
\end{proof}

\begin{rem} \label{rem:comparison_laplacian}
 It is also fairly easy to show that an elliptic operator cannot be symmetric with respect to two different volumes (that is volume forms that differ by more than a constant). As the volume $\Omega^F$ that we use is in general different from the Busemann--Hausdorff volume, we can conclude that our operator is different from Centore's mean-value Laplacian \cite{Cen:mean-value_laplacian}. The same consideration shows that our operator is also different from Bao and Lackey's Laplacian \cite{BaoLackey} (see \cite{moi:these} for the details).
\end{rem}
Note that we are now done with the proof of Theorem \ref{thmintro:Laplacian}.

\section{Energy, Rayleigh quotient and spectrum} \label{sec:energy_and_spectrum}

\begin{defin}
 For any function $u: M \rightarrow \R$\ such that the following makes sense, we define the \emph{Energy} of $u$ by:
\begin{equation}
 E(u) := \frac{n}{\voleucl \left(\S^{n-1}\right) } \int_{HM} \left|L_X\left(\pi^{\ast}u \right)\right|^2 \ada.
\end{equation}
The \emph{Rayleigh quotient} is then defined by
\begin{equation}
 R(u) := \frac{E(u)}{\int_M u^2\, \Omega}.
\end{equation}

\end{defin}

Let us first clarify the space on which those functionals acts; it is a Sobolev space which depends on the manifold. We shall mainly be interested in the case when $M$\ is closed. However, the results described in this section are true for manifolds with (sufficiently smooth) boundary, and we hence consider $M$\ to be compact with possible smooth boundary.\\
We denote by $C^{\infty}_0(M)$\ the space of smooth functions with compact support in the interior of $M$, and we consider the following inner product on it:
$$
\langle u,v \rangle_{1} = \int_M uv \; \Omega + \int_{HM} L_X\left(\pi^{\ast}u \right) L_X\left(\pi^{\ast}v \right)\; \ada.
$$

\begin{defin}
 We let $H^1(M)$\ be the completion of $C^{\infty}_0 (M)$\ with respect to the norm $\rVert \cdot \lVert_{\empty_{1}}$.
\end{defin}

 The energy and Rayleigh quotient are naturally defined on $H^1(M)$. The Finsler--Laplace operator is an unbounded operator on $L^2(M)$ with domain in $H^1(M)$, where $L^2(M)$\ denotes the set of square-integrable functions with respect to the volume $\Omega^F$. A classical embedding theorem (see \cite[Lemma 3.9.3]{Narasimhan}) is

\begin{thm}[Rellich--Kondrachov]
 If $M$\ is compact with smooth boundary, then $H^1(M)$\ is compactly embedded in $L^2(M)$.
\end{thm}

The Energy we defined is naturally linked to the Finsler--Laplace operator:
\begin{thm}
\label{th:min_energy}
 $u \in H^1(M)$\ is a minimum of the energy if and only if $u$\ is harmonic, i.e., $\Delta^F(u) = 0$.
\end{thm}

\begin{proof}
 Let $u,v \in H^1(M)$\ we want to compute $\frac{d}{dt} E(v+tu)$. Let $c_n = \frac{n}{\voleucl\left(\S^{n-1}\right) } $, we have
\begin{equation}
\label{calcul_E_v_tu}
 E(v+tu) = c_n \int_{HM} \! \left(L_X \pi^{\ast}v\right)^2 +2 t L_X \pi^{\ast}v L_X \pi^{\ast}u +t^2 \left(L_X \pi^{\ast}u\right)^2 \; \ada,
\end{equation}
therefore,
\begin{eqnarray*}
 \frac{d}{dt}\left(E(v+tu) \right)_{|_{t=0}} &=& 2 c_n \int_{HM} L_X \pi^{\ast}v L_X \pi^{\ast}u  \, \ada, 
\end{eqnarray*}
and, applying the Finsler--Green formula (Proposition \ref{prop:green_formula}, note that $u \in H^1(M)$\ implies that $u|_{\partial M} =0$\ hence the Finsler--Green formula applies without modifications even when $M$\ has a boundary), we obtain:
\begin{equation*}
 \frac{d}{dt}\left(E(v+tu) \right)_{|_{t=0}} = 2 \int_{HM} u \Delta^F v  \, \Omega^F.
\end{equation*}
So, if $v$\ is harmonic, then it is a critical point of the energy, and \eqref{calcul_E_v_tu} shows that it must be a minimum. Conversely, if $v$\ is a critical point, then for any $u \in H^1(M)$, $\langle \Delta^F v, u \rangle = 0$, which yields $\Delta^F v = 0$. \qedhere

\end{proof}

\subsection{Spectrum}

In this section we solve the Dirichlet eigenvalue problem, i.e., $M$\ is a compact manifold (with or without boundary), and we want to find $u \in C^{\infty}\left(M\right)$\ and $\lambda \in \R$\ such that 
\begin{equation*}
\left\{
\begin{aligned}
 \Delta u + \lambda u &= 0 \:\; \text{on}\; M\\
 u &= 0 \:\; \text{on}\; \partial M.
\end{aligned}
\right.
\end{equation*}
It is well known that the Laplace--Beltrami operator gives rise to an unbounded, strictly increasing sequence of eigenvalues with finite-dimensional pairwise orthogonal eigenspaces (see \cite{BergerGauduchonMazet,Chavel:eigenvalues} for general surveys of spectral problems for the Laplace--Beltrami operator). This stays in particular true for any densely defined, symmetric, positive unbounded operator on the $L^2$ space of a compact manifold (see \cite{Kato:Perturbation_theory,ReedSimon:I,ReedSimon:II} for the general theory). Therefore:

\begin{thm} Let $M$\ be a compact manifold, $F$\ a Finsler metric on $M$.
\label{thm:spectre_discret} 
 \begin{enumerate}
  \item The set of eigenvalues of $-\Delta^F$\ consist of an infinite, unbounded sequence of non-negative real numbers $ \lambda_0 < \lambda_1 < \lambda_2 < \dots$.
  \item Each eigenvalue has finite multiplicity and the eigenspaces corresponding to different eigenvalues are $L^2\left(M, \Omega \right)$-orthogonal.
  \item The direct sum of the eigenspaces is dense in $L^2\left(M, \Omega \right)$\ for the $L^2$-norm and dense in $C^k\left(M\right)$ for the uniform $C^k$-topology.
 \end{enumerate}
\end{thm}

\begin{rem}
 When $M$\ is closed, then $\lambda_0 = 0$\ and the associated eigenfunctions are constant.
\end{rem}

One of the possible proofs of the above theorem uses the Min-Max principle, which also gives an expression for the eigenvalues:
\begin{thm}[Min-Max principle] \label{thm:min_max}
The first eigenvalue is given by
\begin{equation*}
\lambda_0 =  \inf \left\{ R(u) \mid u \in H^1(M) \right\},
\end{equation*}
and its eigenspace $E_0$ is the set of functions realizing the above infimum. The following eigenvalues are given by
\begin{equation*}
 \lambda_k= \inf \left\{ R(u) \mid u \in \bigcap_{i=1}^{k-1} E_i^{\perp} \right\},
\end{equation*}
where their eigenspaces $E_k$ are given by the set of functions realizing the above infimum.\\
In particular, if $M$ is closed, the first non-zero eigenvalue is 
\begin{equation*}
 \lambda_1 = \inf \lbrace R(u) \mid u \in H^1(M), \; \int_M \!\! u \; \Omega = 0 \rbrace.
\end{equation*}

\end{thm}

 The Min-Max principle proof can be found in all generality in \cite{ReedSimon:IV}. For the reader familiar with the Riemannian context, the proof given in \cite{Ber:SpectralGeo} is easily adaptable to the case at hand (the details can be found in \cite{moi:these}).

\begin{rem}
 Although the Min-Max principle gives an expression for the eigenvalues, it is impractical for computations. However it is often used to get bounds on the eigenvalues in general and on the first non-zero eigenvalue in particular. 
\end{rem}

\subsection{Conformal change}

The Energy allows us to give a simple proof that in dimension two, the Laplacian is a conformal invariant.

\begin{thm}
\label{thm:inv_conforme}
Let $(\Sigma,F)$\ be a Finsler surface, $f \colon \Sigma \xrightarrow{C^{\infty}} \R$\ and $F_f = e^f F$. Then,
\begin{equation*}
 \Delta^{F_f} = e^{-2f} \Delta^F.
\end{equation*}

\end{thm}

We first prove the following result:
\begin{prop}
 Let $(M,F)$\ be a Finsler manifold of dimension $n$, $f \colon M \xrightarrow{C^{\infty}} \R$\ and $F_f = e^f F$. Set $E_f$\ the Energy associated with $F_f$. Then, for $u\in H^1\left(M\right)$

\begin{equation*}
 E_f(u) = c_n \int_{HM} e^{(n-2)f} \left(L_X \pi^{\ast}u \right)^2 \, \ada,
\end{equation*}
where $c_n = \frac{n}{\voleucl \left(\S^{n-1}\right)}$.
In particular, when $n=2$\ the Energy is a conformal invariant.

\end{prop}

\begin{proof}
 The subscript $f$\ indicates that we refer to the object associated with the Finsler metric $F_f$. $X_f$\ is a second-order differential equation, so (see \cite{Fou:EquaDiff}) there exist a function $m \colon HM \rightarrow \R$\ and a vertical vector field $Y$\ such that 
$$
X_f = m X + Y.
$$
We have already seen that $A_f = e^f A$\ and that $A_f \wedge dA_f^{n-1} = e^{nf} \ada$. Using $A_f\left(X_f\right)=1$\ and that $VHM$\ is in the kernel of $A$\ we have
\begin{equation*}
 1 = e^f A\left(mX +Y\right) = e^f m A\left(X\right) = e^f m.
\end{equation*}
Now,
\begin{align*}
 E_f(u) &= c_n \int_{HM} \left(L_{X_f} \pi^{\ast} u \right)^2 \, A_f \wedge dA_f^{n-1}, \\
        &= c_n \int_{HM} \left(L_{m X + Y} \pi^{\ast} u \right)^2   e^{nf} \, \ada, \\
        &= c_n \int_{HM}  e^{nf}\left(mL_{ X } \pi^{\ast} u + L_{Y} \pi^{\ast} u \right)^2   \, \ada.
\end{align*}
As $u$\ is a function on the base and $Y$\ is a vertical vector field, $L_{Y} \pi^{\ast} u =0$. So the preceding equation becomes:
\begin{align*}
 E_f(u) &= c_n \int_{HM}  e^{nf} m^2 \left(L_{ X } \pi^{\ast} u  \right)^2   \, \ada, \\
   &= c_n \int_{HM}  e^{(n-2)f}  \left(L_{ X } \pi^{\ast} u  \right)^2   \, \ada. \qedhere
\end{align*}

\end{proof}

\begin{proof}[Proof of Theorem \ref{thm:inv_conforme}]
 Let $u,v \in H^1\left(\Sigma\right)$, we have already shown (Theorem \ref{th:min_energy}) that: $\frac{d}{dt}\left(E(v+tu) \right)_{|_{t=0}} = -2 \int_{\Sigma} u \Delta^{F} v \; \Omega $.\\
The conformal invariance of the Energy yields: for $u,v\in H^1 \left(\Sigma\right)$
\begin{equation*}
 -2 \int_{\Sigma} u \Delta^{F} v \; \Omega = -2 \int_{\Sigma} u \Delta^{F_f} v \; \Omega_f = -2 \int_{\Sigma} e^{2f} u \Delta^{F_f} v \; \Omega,
\end{equation*}
where we used $\Omega_f = e^{2f} \Omega$\ (see equation \eqref{eq:omega1}) to obtain the last equality. We can re-write this last equality as: for $u,v\in H^1\left(\Sigma\right)$
\begin{equation}
 \langle \left(\Delta^{F} - e^{2f} \Delta^{F_f} \right)v , u \rangle = 0,
\end{equation}
which yields the desired result. \qedhere

\end{proof}

\section{Explicit representation and computation of spectrum} \label{sec:explicit_representation}

In this section, we give explicit representations of the Finsler--Laplace operator and its spectrum for Katok--Ziller metrics on the 2-torus and the 2-sphere. We start by describing their construction in a slightly more general context than in \cite{Ziller:GKE}, then obtain an explicit local coordinates formula (Proposition \ref{prop:KZ_explicit}) in order to compute our Finsler--Laplace operator.

\subsection{Katok--Ziller metrics}

Let $M$\ be a closed manifold and $F_0$\ a smooth Finsler metric on $M$. We suppose furthermore that $(M,F_0)$ admits a Killing field $V$, i.e., $V$\ is a vector field on $M$\ that generates a one-parameter group of isometries for $F_0$. We construct the Katok--Ziller metrics in an Hamiltonian setting.\\
Recall that $F_0 \colon TM \rightarrow \R$\ is smooth off the zero-section, homogeneous and strongly convex. Therefore, the Legendre transform associated with $\frac{1}{2} F_0 ^2 $ 
$$
\L_0 := d_v \left(\frac{1}{2} F_0 ^2 \right) \colon TM \rightarrow T^{\ast}M,
$$
where $d_v$\ is the vertical derivative, is a global diffeomorphism and we set $H_0 = F_0 \circ \L_0 ^{-1} \colon T^{\ast}M \rightarrow \R$. 
Note that when $F_0$\ is a Riemannian metric, then $H_0$\ is the dual norm.\\
 $T^{\ast}M$\ is a symplectic manifold with canonical form $\omega$. Any function $H \colon T^{\ast}M \rightarrow \R$\ gives rise to an Hamiltonian vector field $X_H$\ defined by 
\begin{equation*}
 dH(y) = \omega\left(X_H , y \right) \; \text{for } \; y\in TT^{\ast}M.
\end{equation*}
Note that $X_{H_0}$\ describes the geodesics of $F_0$.\\
Define $H_1\colon T^{\ast}M \rightarrow \R$\ by $ H_1(x) = x(V)$ and, for $\eps > 0$, set 
\begin{equation*}
 H_{\eps} = H_0 + \eps H_1.
\end{equation*}
$H_{\eps}$\ is also smooth off the zero-section, homogeneous of degree one and strongly convex for sufficiently small $\eps$. As before, the Legendre transform $\L_{\eps} \colon T^{\ast}M \rightarrow TM$\ associated with $\frac{1}{2} H_{\eps}^2$\ is a global diffeomorphism.
\begin{defin}
 The family of generalized Katok--Ziller metrics on $M$\ associated with $F_0$\ and $V$\ is given by
\begin{equation*}
 F_{\eps} := H_{\eps} \circ \L_{\eps}^{-1}
\end{equation*}

\end{defin}

In \cite{Katok:KZ_metric} Katok took $F_0$\ to be the standard Riemannian metric on $\S^n$, and showed that some of these metrics had only a finite number of closed geodesics. In fact, if $\frac{\eps}{2\pi}$\ is irrational then $\S^{2k}$\ and $\S^{2k+1}$\ with their Katok--Ziller metric has $2k$\ closed geodesics \cite{Ziller:GKE}. Bangert and Long \cite{MR2563691} showed that every Finsler metric on $S^2$\ has at least $2$\ closed geodesics, it is still unknown in higher dimension. However, it is in sharp contrast with the Riemannian case and we can wonder whether this should reflect on the spectrum.\\
We only need local coordinate formulas for the Katok--Ziller metrics on the torus and the sphere, but we can give a general formula when $F_0$\ is Riemannian. This result is not new (see Rademacher \cite{Rademacher:Sphere_theorem} for the Katok--Ziller metric on the sphere) and was communicated to us in these more general form by P. Foulon.
\begin{prop}
\label{prop:KZ_explicit}
 Let $F_0 = \sqrt{g}$\ be a Riemannian metric on $M$, $V$\ a Killing field on $M$, and $F_{\eps}$\ the associated Katok--Ziller metric. Then
\begin{equation*}
 F_{\eps}(x,\xi) = \frac{1}{1- \eps^2 g\left(V,V\right)} \left[ \sqrt{g\left(\xi,\xi \right) \left(1-\eps^2 g\left(V,V\right) \right) + \eps^2 g\left(V,\xi\right)^2} -\eps g\left(V,\xi\right) \right].
\end{equation*}

\end{prop}

\begin{rem}
 This formula also shows that if $F_{0}$\ is Riemannian then $F_{\eps}$ is a Randers metric (see \cite{BaoChernShen}).
\end{rem}

\begin{proof}
Let $x\in M$. We choose normal coordinates $\xi_i$\ on $T_xM$ and write $p^i$ the associated coordinates on $T_x^{\ast}M$. We have    $F_0^2\left(x,\xi\right) = \sum \xi_i^2$ and for $p \in T_x^{\ast}M$:
$$
H_0 \left(x,p\right) = ||p|| = \sqrt{\sum (p^i)^2}.
$$
$H_{\eps}$\ is then given by: $ H_{\eps}\left(x,p\right) = H_0 \left(x,p\right) + \eps H_1 \left(x,p\right) = ||p|| + \eps \left<p | V\right>$.
Recall that $F_{\eps}\left(x, \xi \right) = H_{\eps} \circ \Le^{-1} \left(x, \xi \right)$ where $\Le= d_v\left(\frac{1}{2} H_{\eps}^2 \right) \colon T^{\ast}M \rightarrow TM$.\\
As $\frac{\partial}{\partial x_i}$\ is a vectorial basis of $T_xM$, we can write: 
\begin{align*}
 \Le \left(x,p\right) &= \frac{\partial}{\partial p^i}\left(\frac{1}{2} H_{\eps}^2 \right) \frac{\partial}{\partial x_i} \\
      &= \frac{\partial}{\partial p^i}\left[ \frac{1}{2} ||p||^2 + \eps ||p|| \left<p | V\right> + \frac{\eps^2}{2} \left<p | V\right>^2 \right] \frac{\partial}{\partial x_i} \\
      &= \left[ p^i + \eps \left(\frac{p^i}{||p||} \left<p | V\right> + \eps ||p|| V_i \right) + \eps^2 V_i \left<p | V\right> \right] \frac{\partial}{\partial x_i} \\ 
     &=  \left( p^i + \eps ||p|| V_i \right) \left( 1 + \frac{\eps}{||p||} \left< p | V\right>\right) \frac{\partial}{\partial x_i}. 
\end{align*}
Set $u := \frac{p}{||p||} $, $H_{\eps} \left(x,p\right) = F_{\eps}\left( \Le \left(x,p\right) \right)$\ implies:
\begin{align*}
  ||p|| + \eps \left<p | V\right> &=  F_{\eps}\left(||p|| \left( u^i + \eps  V_i \right) \left( 1 + \frac{\eps}{||p||} \left< p | V\right>\right) \frac{\partial}{\partial x_i} \right) \\
    &= \left(||p|| + \eps \left<p | V\right> \right) F_{\eps}\left( \left( u^i + \eps  V_i \right) \frac{\partial}{\partial x_i} \right). \\
\end{align*}
So $ F_{\eps}\left( \left( u^i + \eps  V_i \right) \frac{\partial}{\partial x_i} \right) = 1$. Set $ \xi = \Le \left(x,p\right)$, we showed that $\xi_i =  F_{\eps}\left(x,\xi\right) \left( u + \eps V \right)_i$ for all $i$. Therefore,
\begin{equation*}
  \left< u |V \right> = \frac{1}{F_{\eps}\left(x,\xi\right)} \left< \xi |V \right> - \eps ||V||^2,
\end{equation*}
and
\begin{align*}
 ||\xi||^2 &=  F_{\eps}^2\left(x,\xi\right) \left[ ||u||^2 + 2 \eps \left< u | V\right> + \eps^2 ||V||^2 \right] \\
          &= F_{\eps}^2\left(x,\xi\right)\left[ 1  + 2 \eps  \frac{ \left< \xi |V \right> }{F_{\eps}\left(x,\xi\right)}  -2\eps^2 ||V||^2+ \eps^2 ||V||^2 \right] .
\end{align*}
In order to get $F_{\eps}\left(x,\xi\right) $ we solve the equation
\begin{equation*}
  F_{\eps}^2\left(x,\xi\right) \left(1 - \eps^2 ||V||^2 \right)  + 2 \eps   \left< \xi |V \right>  F_{\eps}\left(x,\xi\right) - ||\xi||^2 = 0,
\end{equation*}
and obtain:
\begin{equation*}
  F_{\eps}\left(x,\xi\right) = \frac{ -  \eps   \left< \xi |V \right> + \sqrt{\eps^2  \left< \xi |V \right>^2 + \left(1 - \eps^2 ||V||^2 \right)||\xi||^2 }}{\left(1 - \eps^2 ||V||^2 \right)}. \qedhere
\end{equation*}

\end{proof}

Before getting on to the examples, we want to point out the following (unpublished) result on the Katok--Ziller examples:
\begin{thm}[Foulon \cite{Fou:perso}]
 The flag curvatures of the family of Katok--Ziller metrics are constant.
\end{thm}

\begin{rem}
By the classification result of \cite{BaoRoblesShen}, the Katok--Ziller metrics are the only Randers metrics on $\S^2$\ of constant flag curvature.\\
We cannot use the Katok--Ziller construction for negatively curved surface as a compact hyperbolic surfaces never admits a one-parameter group of isometries.
\end{rem}

\subsection{On the $2$-Torus}

We set $\T  = \R^2 /\Z^2$, $(x,y)$\ (global) coordinates on $\T$\ and $\left( \xi_x, \xi_y \right)$\ local coordinates on $T_p\T$. Let $\eps$\ be a small parameter, the Katok--Ziller metric on $\T$\ associated with the standard metric and to the Killing field $V = \frac{\partial}{\partial x}$, is given by:
$$
F_{\eps}(x,y; \xi_x, \xi_y ) = \frac{1}{1-\eps^2} \left( \sqrt{\xi_x^2 + (1-\eps^2) \xi_y^2} - \eps \xi_x \right).
$$

\begin{thm}
\label{thm:KZ_torus}
The Finsler--Laplace operator for $(\T^2,F_{\eps})$ is
 $$
\Delta^{F_{\eps}} = \frac{2 \left(1-\eps^2\right)}{1 + \sqrt{1- \eps^2}} \left( \sqrt{1- \eps^2} \parxx + \paryy \right)
$$
and the spectrum is the set of $\lambda_{(p,q)}$, $(p,q)\in \Z^2$ such that:
\begin{equation*}
\lambda_{(p,q)} = 4\pi^2 \frac{2 \left(1-\eps^2\right)}{1 + \sqrt{1- \eps^2}} \left( \sqrt{1-\eps^2} p^2 + q^2 \right).
\end{equation*}

\end{thm}

\begin{rem} \label{rem:utilisation_unique1}
\begin{itemize}
 \item $\T$\ with the Katok--Ziller metric is "iso-Laplace" to the flat torus equipped with the symbol metric, i.e.,it is obtained as the quotient of $\R^2$\ by the lattice $A\Z \times B\Z$, where 
\begin{equation*}
 A^2 =  \frac{1+ \sqrt{1-\eps^2}}{2\left(1-\eps^2\right)^{3/2}} \quad \text{and} \quad B^2 = \frac{1 + \sqrt{1- \eps^2}}{2\left(1-\eps^2\right)}.
\end{equation*}
 \item Because of Lemma \ref{lem:existence_unicity}, to show that the Finsler--Laplace operator is the Laplace--Beltrami operator of the symbol metric, it is enough to show that the Finsler volume form is a constant multiple of the Riemannian volume. In fact, we can show that \emph{any} Killing field generates a Katok--Ziller metric on $\T$ which is iso-Laplace to the Laplace--Beltrami operator of the associated symbol metric (see \cite{moi:these}). However, for the sake of simplicity, we give the actual computations only in the above case.
\end{itemize}
 
\end{rem}
Note that for any flat Riemannian torus, the Poisson formula gives a link between the eigenvalues of the Laplacian and the length of the periodic orbits. In the case at hand we lose this relationship as there is a priori no link between the length of the periodic geodesics for the Finsler metric and the length of the closed geodesics in the isospectral torus.\\

\begin{proof}

 \textbf{Vertical derivative and coordinate change.}\\
 In the local coordinates $\left(x,y, \xi_x, \xi_y\right)$\ on $T\T$ we have:
\begin{equation*}
 d_v F_{\eps}= \frac{1}{1-\eps^2} \left( f_x dx + f_y dy \right),
\end{equation*}
where
\begin{equation*}
 f_x := \frac{\xi_x}{\sqrt{\xi_x^2 + (1-\eps^2) \xi_y^2}} - \eps  \quad \text{and} \quad  f_y := \frac{\xi_y}{\sqrt{\xi_x^2 + (1-\eps^2) \xi_y^2}}.
\end{equation*}
We choose a local coordinate system $\left(x,y,\theta\right)$\ on $H\T$ where $\theta$\ is determined by 
\begin{equation*}
\left\{
\begin{aligned} 
 \cos \theta & =  f_x + \eps \\
 \sin \theta & =  \frac{f_y}{\sqrt{1-\eps^2}},
\end{aligned}
\right.
\end{equation*}
As the Hilbert form $A$\ is the projection on $H\T$\ of the vertical derivative of $F$, we have:
\begin{equation*}
 A = \frac{1}{1-\eps^2} \left( \left(\cos \theta -\eps \right) dx + \sqrt{1-\eps^2}\sin \theta dy \right).
\end{equation*}

\textbf{Liouville volume and angle form.}\\
 We have:
\begin{align*}
 dA & = \frac{1}{1-\eps^2} \left( -\sin \theta d\theta \wedge dx + \sqrt{1-\eps^2} \cost d\theta \wedge dy \right) \\
A \wedge dA & = \left(\frac{1}{1-\eps^2}\right)^{\frac{3}{2}} \left(-1 + \eps \cos \theta \right) d\theta \wedge dx \wedge dy.
\end{align*}
Therefore $ \alpha = \left(1 - \eps \cost\right) d\theta$.

\textbf{Geodesic flow.}\\
Let $X = X_x \parx + X_y \pary + X_{\theta} \partheta$ be the geodesic flow, equation \eqref{eq:Reeb_field} is equivalent to:
\begin{equation*}
\left\{
\begin{aligned}
X_{\theta} &= 0  \\
\sint X_x -  \sqrt{1-\eps^2} \cost X_y &= 0 \\
\left(\cost -\eps \right) X_x +  \sqrt{1-\eps^2} \sint X_y &= 1-\eps^2 
\end{aligned}
\right.
\end{equation*}
Hence $X_x = \frac{1-\eps^2}{1- \eps \cost} \cost$, $X_y = \frac{\sqrt{1-\eps^2}}{1- \eps \cost} \sint$ and $X_{\theta} = 0$.

\textbf{The Laplacian.}\\
The second Lie derivative of $X$\ is $L_X^2  = X_x^2 \parxx + X_y^2 \paryy + X_x X_y \parxy$. So, for $p\in S$
\begin{equation*}
 \Delta^{\eps} = \frac{1}{\pi}\left( \int_{H_pS} X_x^2 \alpha \parxx + \int_{H_pS} X_y^2 \alpha \paryy + \int_{H_pS} X_x X_y \alpha \parxy \right).
\end{equation*}
As $X_x$\ and $X_y$\ are of different parity (in $\theta$) we have $\int_{H_pS} X_x X_y \alpha =0 $. Hence
\begin{equation*}
 \Delta^{\eps} = \frac{1}{\pi}\left( \int_{H_pS} X_x^2 \alpha \parxx + \int_{H_pS} X_y^2 \alpha \paryy \right).
\end{equation*}
Direct computation gives:
\begin{align*}
 \int_{H_pS} X_x^2 \alpha &= 2 \pi \frac{\left(1-\eps^2\right)^{\frac{3}{2}}}{1 + \sqrt{1- \eps^2}},\\
 \int_{H_pS} X_y^2 \alpha &= 2 \pi \frac{1-\eps^2}{1 + \sqrt{1- \eps^2}}.
\end{align*}
Therefore, the Finsler--Laplace operator is given by 
\begin{equation*}
 \Delta^{\eps} = \frac{2 \left(1-\eps^2\right)^{\frac{3}{2}}}{1 + \sqrt{1- \eps^2}} \parxx + \frac{2 \left(1-\eps^2 \right)}{1 + \sqrt{1- \eps^2}} \paryy.
\end{equation*}

\textbf{The spectrum.}\\
To compute the spectrum we consider Fourier series of functions on $\T$.\\
Any function $f \in C^{\infty }(\T)$\ can be written as:
$$
f(x,y) = \sum_{(p,q)\in \Z^2} c_{(p,q)} e^{2i\pi (px+qy)}
$$
and we are lead to solve:
\begin{equation}
\label{eq:spectre_fourier}
 \Delta^{F_{\eps}} f + \lambda f = \sum_{(p,q)\in \Z^2} c_{(p,q)} \left[ -4\pi^2 \left(a p^2 + b q^2\right) +\lambda \right] e^{2i\pi (px+qy)} = 0 
\end{equation}
where
\begin{equation*}
 a = \frac{2 \left(1-\eps^2\right)^{3/2}}{1 + \sqrt{1- \eps^2}} \quad \text{and}\quad b = \frac{2 \left(1-\eps^2\right)}{1 + \sqrt{1- \eps^2}}.
\end{equation*}

Now, for any $(p,q) \in \Z^2$, $\lambda_{(p,q)} = 4\pi^2 \frac{2 \left(1-\eps^2\right)}{1 + \sqrt{1- \eps^2}} \left( \sqrt{1-\eps^2} p^2 + q^2 \right)$ is a solution to \eqref{eq:spectre_fourier}.
\end{proof}

\subsection{On the $2$-Sphere}

We set $\S^2= \lbrace (\phi,\theta) \mid \phi \in \left[0,\pi\right], \; \theta \in \left[0,2 \pi\right] \rbrace$ and take $\left(\phi,\theta; \xi_{\phi}, \xi_{\theta} \right)$\ the associated local coordinates on $T\S^2$. The Katok--Ziller metric associated with the standard metric and to the Killing field $V = \sinp \frac{\partial}{\partial \theta}$\ is given by:
\begin{equation*}
 F_{\eps} \left(\phi,\theta; \xi_{\phi}, \xi_{\theta} \right) = \frac{1}{1-\ee} \left( \sqrt{\left(1-\ee\right)\xi_{\phi}^2 + \sinpp \xi_{\theta}^2} - \eps\sin^2(\phi) \xi_{\theta}\right),
\end{equation*}

\begin{thm}
\label{th:laplacien_S2}
 The Finsler--Laplace operator on $(\mathbb{S}^2, F_{\eps})$ is given by:
\begin{multline} \label{eq:laplacien_S2}
    \Delta^{F_{\eps}} = \frac{2}{1+\sqrt{1-\ee}} \Biggl[ \frac{1}{\sinpp} \left(1-\ee \right)^{3/2} \parttheta  \\
              + \left(1-\ee \right) \parpphi  +\frac{\cosp}{\sinp} \left(\ee + \sqrt{1-\ee}  \right)\parphi \Biggr].
\end{multline} 
\end{thm}

Note that if we compute the Laplace--Beltrami operator for the symbol metric, we can see that $\Delta^{F_{\eps}}$ is \emph{not} Riemannian, hence the question of whether this operator is iso-spectral to a Riemannian one is non-trivial contrarily to the torus case.\\
Unfortunately, the complexity of this formula dampened our hopes of finding an explicit expression of the spectrum. However, we can still find the first eigenvalue and give an approximation of the others.\\
The spectrum of the Laplace--Beltrami operator on $\S^2$ is $\lbrace -l(l+1) \mid l \in \N \rbrace$ and an eigenspace is span by functions $\ylm$ with $m\in \Z$ such that $-l\leq m \leq l$. These functions are called \emph{spherical harmonics} and are defined by:
\begin{equation*}
\ylm\left(\phi, \theta\right) := e^{i m \theta} P_l^m\left(\cosp\right),
\end{equation*}
where $P_l^m$\ is the associated Legendre polynomial.\\
 We can see clearly from formula \eqref{eq:laplacien_S2} that when $\eps$\ tends to $0$\ we obtain the usual Laplace--Beltrami operator on $\S^2$, we will therefore look for eigenfunctions close to the spherical harmonics. It turns out that the $Y_1^m$\ are eigenfunctions of $\Delta^{F_{\eps}}$\ for any  $\eps$, which yields Theorem \ref{thmintro:lambda1}:

\begin{cor}
\label{cor:lambda1}
 The smallest non-zero eigenvalue of $-\Delta^{F_{\eps}}$\ is 
\begin{equation}
 \lambda_1 = 2 -2 \eps^2 = \frac{8 \pi}{\text{vol}_{\Omega}\left(\S^2\right)}
\end{equation}
It is of multiplicity two and the eigenspace is generated by $Y_1^1$\ and $Y_1^{-1}$.
\end{cor}

The fact that we have the above formula for $\lambda_1$ is quite interesting; first, it shows us that there do exist relationships between some geometrical data associated with a Finsler metric (here the volume) and the spectrum of the Finsler--Laplace operator. Secondly, remember the following result:
\begin{thm*}[Hersch \cite{Hersch}]
 For any Riemannian metric $g$\ on $\S^2$,
$$
\lambda_1 \leq \frac{8\pi}{\text{vol}_{g}\left(\S^2\right)}.
$$
Furthermore, the equality is realized only in the constant curvature case.
\end{thm*}
So the Katok--Ziller metrics on $\S^2$ give us a continuous family of metrics realizing that Riemannian maximum!\\

Note that we also have $\Delta^{F_{\eps}} Y_1^0 = -2 Y_1^0$. However, the $\ylm$\ with $l\geq 2$\ are no longer eigenfunctions of $\Delta^{F_{\eps}}$. This is probably related to the breaking of the symmetries that the Katok--Ziller metrics induce.\\
 In the following, if $m$\ happens to be greater than $l$, we set $\ylm =0$. We denote by $\langle \cdot , \cdot \rangle$\ the inner product on $L^2\left(\S^2\right)$ defined by:
\begin{equation*}
 \langle f,g \rangle = \int_0^{2 \pi} \int_0^{\pi} f \bar{g} \sinp d\phi d\theta.
\end{equation*}

\begin{thm}
\label{thm:spectre_S2}
 Let $f$\ be an eigenfunction for $\Delta^{F_{\eps}}$ and $\lambda$\ its eigenvalue. There exist unique $l$\ and $m$\ in $\N$, $0\leq m \leq l$, such that $f = a \ylm + b Y_l^{-m} + g$, where $g$\ uniformly tends to $0$\ with $\eps$, and 
\begin{multline}
\label{eq:spectre_S2}
  \lambda =  -l(l+1) + \eps^2 \Biggl[ \frac{m^2}{2\left(2l-1 \right)} \left( 2\left(l+1\right) + \frac{3 l\left(l-1\right)}{ \left(2l+3 \right)} \right)   \\
             +\frac{3l\left(l-1\right)}{2\left(2l-1 \right)}\left( 1 +  \frac{l^2+l-1}{ \left(2l+3 \right)\left(2l-1 \right)} \right) \Biggr] +o\left(\eps^2 \right).
\end{multline}

\end{thm}

Note that the Katok--Ziller transformation gets rid of most of the degeneracy of the spectrum. If $\eps\neq 0$, the eigenvalues are at most of multiplicity two, and are of multiplicity $2l+1$\ if $\eps$\ is zero.\\
We can state even more on the multiplicity of eigenvalues. Set:
\begin{equation*}
\begin{array}{rrcl}
\Psi  \colon  & \S^2          &\longrightarrow& \S^2 \\
       & (\phi, \theta ) &\longmapsto& (\pi - \phi, -\theta)
\end{array}
\end{equation*}
Theorem \ref{th:laplacien_S2} implies that $\Delta^{F_{\eps}}$\ is stable by $\Psi$ i.e., for any $g$, $\left(\Delta^{F_{\eps}} g \right) \circ \Psi = \Delta^{F_{\eps}} \left( g \circ \Psi \right)$.
So if $f$\ is an eigenfunction for $\lambda$\ then $f \circ \Psi$\ also. Therefore, either the subspace generated by $f$\ is stable by $\Psi$\ or $\lambda$\ is of multiplicity at least (and hence exactly) two.

\begin{rem}
When $\eps >0$, $F_{\eps}$\ is not preserved by $\Psi$.
\end{rem}

\subsubsection{Proof of Theorem \ref{th:laplacien_S2}}
This proof follows the same lines as Theorem \ref{thm:KZ_torus}, the computations being more involved and a bit lengthy. We just give the main steps.
 \textbf{Vertical derivative and change of coordinates.}

Set $g_{\eps}\left(\phi,\theta; \xi_{\phi}, \xi_{\theta} \right) = \left(1-\ee\right)\xi_{\phi}^2 + \sinpp \xi_{\theta}^2$. We have 
$$
d_v F_{\eps} = \frac{\partial F_{\eps}}{\partial \xi_{\phi}} d\phi + \frac{\partial F_{\eps}}{\partial \xi_{\theta}} d\theta
$$
 where $  \dfrac{\partial F_{\eps}}{\partial \xi_{\phi}} = \dfrac{\xi_{\phi}}{\sqrt{g_{\eps}}}$ and $\dfrac{\partial F_{\eps}}{\partial \xi_{\theta}} =  \dfrac{1}{1-\ee}\left(\dfrac{\xi_{\theta} \sinpp}{\sqrt{g_{\eps}}} - \eps \sinpp \right)$.\\
From now on we consider the local coordinate $\psi \in \left[0,2\pi\right]$\ on $H_{(\phi,\theta)}\S^2$, defined by,
\begin{equation*}
\left\{
\begin{aligned} 
\cospsi& =  \frac{\xi_{\theta} \sinp}{\sqrt{g_{\eps}}} \\
 \sinpsi& = \sqrt{1-\ee} \frac{\xi_{\phi}}{\sqrt{g_{\eps}}} 
\end{aligned}
\right.
\end{equation*}
\textbf{Hilbert form and Liouville volume.}

The Hilbert form $A$ associated with $F_{\eps}$\ is given by
\begin{equation*}
 A = \frac{1}{1- \ee}\left( f_1 d\phi + f_2 d\theta \right),
\end{equation*}
with $f_1 = \sqrt{1-\ee} \sinpsi$ and $f_2 = \sinp\cospsi - \eps \sinpp$. In order to simplify the computations, note that $f_1$\ is odd in $\psi$, $f_2$\ is even and they do not depend on $\theta$. The exterior derivative of $A$\ is given by:

\begin{equation*}
 dA = \frac{1}{1- \ee} \left(\parfunpsi d\psi\wedge d\phi + \parfdepsi d\psi\wedge d\theta + f_3 d\phi\wedge d\theta \right).
\end{equation*}
where
\begin{align*}
 \parfunpsi &= \sqrt{1-\ee} \cospsi, \\
 \parfdepsi &= -\sinp \sinpsi, \\
 f_3 &= \cosp \frac{\cospsi -2 \e +\ee \cospsi}{1-\ee}.
\end{align*}
Therefore
\begin{equation}
\label{eq:ada_sphere}
 A \wedge dA = \frac{\sinp}{\left(1-\ee \right)^{3/2}} \left( 1- \e \cospsi \right) d\psi \wedge d\phi \wedge d\theta.
\end{equation}
Using the construction of the angle form $\alpha$\ (section \ref{par:construction}), we obtain:
\begin{equation}
\label{eq:angle_sphere}
 \alpha = \left(1- \e \cospsi\right) d\psi
\end{equation}
\textbf{Geodesic flow.}

Let $X = \xpsi \parpsi + \xtheta \partheta + \xphi \parphi$\ be the geodesic flow of $F_{\eps}$. Equation \eqref{eq:Reeb_field} gives the system:
\begin{equation*}
\left\{
 \begin{aligned}
  \parfunpsi \xphi + \parfdepsi \xtheta &= 0 \\
 \xphi f_3+ \xpsi \parfdepsi &= 0 \\
 - \xtheta f_3 + \xpsi \parfunpsi &= 0 \\
 f_1 \xphi + f_2 \xtheta &= 1-\ee
 \end{aligned}
\right. .
\end{equation*}
Which yields
\begin{align*}
\xtheta &= \frac{1-\ee}{\sinp} \; \frac{\cospsi}{1-\e \cospsi} ,\\
\xphi &= \sqrt{1-\ee} \; \frac{\sinpsi}{1-\e \cospsi} ,\\
\xpsi &= \frac{1}{\sqrt{1-\ee}} \; \frac{\cosp}{\sinp} \; \frac{\cospsi -2\e + \ee \cospsi}{1-\e\cospsi}.
\end{align*}
\textbf{The Finsler--Laplace operator.}

Let $f \colon \S^2 \rightarrow \R$. We start by computing $L_X^2 \pi^{\ast}f$.\\
As $\parpsi\left(\pi^{\ast}f\right) = 0 $\ and that $X$\ does not depend on $\theta$\, we get:

\begin{multline*}
 L_X^2 \pi^{\ast}f = \xtheta^2 \frac{\partial^2 f}{\partial \theta^2} + \xtheta\xphi \frac{\partial^2 f}{\partial \phi \partial \theta}  + \xphi \xtheta \frac{\partial^2 f}{\partial \theta \partial \phi } + \xphi \frac{\partial \xtheta }{\partial \phi} \frac{\partial f}{\partial \theta}  \\
+ \xphi \frac{\partial \xphi }{\partial \phi} \frac{\partial f}{\partial \phi} + \xphi^2 \frac{\partial^2 f}{\partial \phi^2} + \xpsi \frac{\partial \xtheta }{\partial \psi} \frac{\partial f}{\partial \theta}  + \xpsi \frac{\partial \xphi }{\partial \psi} \frac{\partial f}{\partial \phi}.
\end{multline*}
Since we are only interested in $\int_{H_x\S^2} L_X^2 \pi^{\ast}f \alpha$, we can use the parity properties (with respect to $\psi$) of the functions $\xtheta, \; \xphi$\ and $\xpsi$\ (which are respectively even, odd and even) to get rid of half of the above terms. We obtain:
\begin{multline*}
 \pi \Delta^{F_{\eps}} f(p) =  \int_{H_p\S^2} \hspace{-1,5mm}\xtheta^2 \; \alpha \; \frac{\partial^2 f}{\partial \theta^2} + \int_{H_p\S^2} \hspace{-1,5mm} \xphi^2 \; \alpha \; \frac{\partial^2 f}{\partial \phi^2} \\
   + \int_{H_p\S^2} \left( \xpsi \frac{\partial \xphi }{\partial \psi} + \xphi \frac{\partial \xphi }{\partial \phi} \right) \alpha \; \frac{\partial f}{\partial \phi}.
\end{multline*}

Direct computation (with a little help from Maple) yields:
\begin{multline*}
   \Delta^{F_{\eps}} = \frac{2 \left(1-\ee \right)^{\frac{3}{2}}}{\sinpp \left(1+\sqrt{1-\ee} \right)} \; \frac{\partial^2 }{\partial \theta^2}  + 2  \frac{1-\ee}{1+\sqrt{1-\ee}} \; \frac{\partial^2 }{\partial \phi^2}  \\
 + \frac{2 \cosp}{\sinp} \left( 2 - \frac{1}{1+ \sqrt{1-\ee}} -\sqrt{1-\ee} \right) \frac{\partial }{\partial \phi}.
\end{multline*}
This concludes the proof of Theorem \ref{th:laplacien_S2}.

\subsubsection{Proof of Theorem \ref{thm:spectre_S2}}

We state the following property of spherical harmonics that will be useful in later computations:
\begin{prop}
\label{prop:ylm}
 Let $l\in \N$, and $m\in \Z$, such that $|m|\leq l$, then the associated Legendre polynomial $P_l^m\left(\cosp\right)$, denoted here by $\plm$, is a solution to the equation:
\begin{equation}
 \frac{\partial^2 \plm}{\partial \phi^2} + \frac{\cosp}{\sinp}\frac{\partial \plm}{\partial \phi} +\left(l(l+1) - \frac{m^2}{\sinpp} \right) \plm = 0 \label{eq:legendre},
\end{equation}
They verify (see \cite[formulas 8.5.3 to 8.5.5]{MR1225604}):
\begin{subequations}
\begin{align}
 \left(2l-1\right) \cosp \tilde{P}^{m}_{l-1} &= (l-m)\plm + \left(l+m-1\right) \tilde{P}^{m}_{l-2} ,\label{eq:rec_legendre1} \\
\sinp \frac{\partial \plm}{\partial \phi} & = l \cosp \plm - (l+m) \tilde{P}^{m}_{l-1} \label{eq:derivative_legendre}, \\
 \sinp \plm &= \frac{1}{2l+1}\left( \tilde{P}_{l-1}^{m+1} - \tilde{P}_{l+1}^{m+1}\right) \label{eq:rec_legendre2}.
\end{align}
\end{subequations}
 The spherical harmonics are an orthogonal Hilbert basis of $L^2\left(\S^2\right)$\ and their norm is given by:
\begin{equation}
 ||\ylm || = \sqrt{\frac{4\pi}{2l+1} \frac{\left(l+m\right)!}{\left(l-m\right)!}} \label{eq:norm_ylm}.
\end{equation}

\end{prop}

We can now proceed with the proof. Take $f$\ an eigenfunction of $\Delta^{F_{\eps}}$ and $\lambda$\ the associated eigenvalue. As the $\ylm$\ form an Hilbert basis of $L^2\left(\S^2\right)$, there exist $a_l^m$\ such that:
$$
f = \sum_{l=0}^{+\infty} \sum_{|m|\leq l} a_l^m \ylm,
$$
where the convergence is a priori in the $L^2$-norm. The elliptic regularity theorem implies that $f \in C^{\infty}\left(\S^2 \right)$, therefore the convergence above is uniform. So $\Delta^{F_{\eps}} f = \sum_{l=0}^{+\infty} \sum_{|m|\leq l} a_l^m \Delta^{F_{\eps}} \ylm$.\\ Let $l,m$\ be fixed, the equation $\langle \Delta^{F_{\eps}} f, \ylm \rangle = \lambda \langle f, \ylm \rangle$ yields:
\begin{equation}
\label{eq:lambda_sphere}
 \lambda a_l^m \lVert \ylm \rVert^2 = \sum_{k=0}^{+\infty} \sum_{|n|\leq k} a_k^n \langle \ylm , \Delta^{F_{\eps}} Y_k^n \rangle.
\end{equation}

\begin{claim}
For any $l,m$\ we have:
 \begin{multline}
\label{eq:Delta_ylm}
 \Delta^{F_{\eps}} \ylm =   -l(l+1) \ylm \\
 + \frac{\eps^2}{\left(1+\sqrt{1-\ee}\right)^2}  \Biggl[  \left(1+2\sqrt{1-\ee}\right)l\left( l-1\right) \sinpp \ylm  \\
  +  \left(2 m^2 \left(1-\ee\right)\right)\ylm  + 2 \frac{l^2 + m^2 +l}{2l+1} \left(1+2\sqrt{1-\ee}\right) \ylm \\
 - 2(l+m)(l+m-1)\left(1+2\sqrt{1-\ee}\right) Y_{l-2}^m \Biggr].
\end{multline}
\end{claim}

The proof is just a computation using Proposition \ref{prop:ylm}.\\
Using the claim, equation \eqref{eq:lambda_sphere} becomes:
\begin{equation*}
 \lambda a_l^m \lVert \ylm \rVert^2 = \sum_{k=0}^{+\infty} a_k^m \langle \ylm , \Delta^{F_{\eps}} Y_k^m \rangle.
\end{equation*}
Now, we can use an expansion of $\Delta^{F_{\eps}} Y_k^m $\ in powers of $\eps$.

\begin{claim}
 For any $l,m$, we have:
 \begin{multline}
 \Delta^{F_{\eps}} \ylm  =  -l(l+1) \ylm + \eps^2 \Biggl[\frac{3 l (l+1)}{4} \sinpp\ylm \\
       + \left( \frac{m^2}{2} + \frac{3\left(l(l+1) +m^2\right)}{ 2l+1} \right) \ylm  + \frac{3}{2} (l+m)(l+m-1) Y_{l-2}^m \Biggr] + O\left(\eps^4\right).
\end{multline}
\end{claim}
The claim follows once again from a straightforward computation.\\
Using this second claim and the orthogonality of the spherical harmonics, equation \eqref{eq:lambda_sphere} now reads:
\begin{multline}
\label{eq:expansion_de_lambda}
 \lambda a_l^m \lVert \ylm \rVert^2 =  -l(l+1) a_l^m \lVert \ylm \rVert^2 + a_l^m \eps^2 \Biggl[\frac{3 l (l+1)}{4} \langle \sinpp\ylm , \ylm \rangle \\
       + \left( \frac{m^2}{2} + \frac{3\left(l(l+1) +m^2\right)}{ 2l+1} \right) \lVert \ylm \rVert^2 \Biggr]  \\
 + \sum_{k \neq l} a_k^m \eps^2  \Biggl[\frac{3 k (k+1)}{4} \langle \sinpp Y_k^m, \ylm \rangle    + \frac{3}{2} (k+m)(k+m-1) \langle Y_{k-2}^m, \ylm \rangle \Biggr] + O\left(\eps^4\right).
\end{multline}

\begin{claim}
 There exist at most one $l$\ such that $\frac{1}{a_l^m}$\ is bounded independently of $\eps$.
\end{claim}

\begin{proof}
 The equation \eqref{eq:expansion_de_lambda} shows that, if $\frac{1}{a_l^m}$\ is bounded as $\eps$\ tends to $0$, then $\lambda$\ tends to $-l(l+1)$, therefore we can have only one such $l$.
\end{proof}

Let $l$\ be given by the previous claim, \eqref{eq:expansion_de_lambda} reduces to:
\begin{multline*}
 \lambda  = -l(l+1)  + \frac{\eps^2}{\lVert \ylm \rVert^2 } \Biggl[\frac{3 l (l+1)}{4} \langle \sinpp\ylm , \ylm \rangle \\
        + \left( \frac{m^2}{2} + \frac{3\left(l(l+1) +m^2\right)}{ 2l+1} \right) \lVert \ylm \rVert^2 \Biggr] + o\left(\eps^2\right).
\end{multline*}
Some more computations (using equations \eqref{eq:rec_legendre2}, \eqref{eq:norm_ylm} and the orthogonality of the spherical harmonics) give:
\begin{equation*}
 \frac{\langle \sinpp\ylm , \ylm \rangle}{\lVert \ylm \rVert^2 }  = 2\frac{l^2+l-1+ m^2}{ \left(2l+3 \right)\left(2l-1 \right)}.
\end{equation*}
So that:
\begin{multline} \label{eq:lambda_derniere}
 \lambda  =  -l(l+1)  + \eps^2 \Biggl[\frac{3 l (l+1)}{2} \frac{l^2+l-1+ m^2}{ \left(2l+3 \right)\left(2l-1 \right)} \\
        + \left( \frac{m^2}{2} + \frac{3\left(l(l+1) +m^2\right)}{ 2l+1} \right) \Biggr] + o\left(\eps^2\right).
\end{multline}
From this equation, we deduce:
\begin{claim}
 There can only be one $m$\ such that $\frac{1}{a_l^m}$\ or $\frac{1}{a_l^{-m}}$\ is bounded independently of $\eps$.
\end{claim}

\begin{proof}
 Otherwise, we would find two different coefficients in $\eps^2$\ for $\lambda$.
\end{proof}

We sum up what we proved: There exists unique $l,m\in \N $, $a,b \in \C$ and $g \colon \S^2 \rightarrow \C$\ such that:
\begin{equation*}
f = a \ylm + b Y_l^{-m} + g 
\end{equation*}
Furthermore, for any $p\in \S^2$, $|g(p)|$\ tends to $0$\ with $\eps$ and the associated eigenvalue verifies equation \eqref{eq:lambda_derniere}. That is, we proved Theorem \ref{thm:spectre_S2}.

\subsubsection{First eigenvalue and volume}

We finish by proving Corollary \ref{cor:lambda1}. Recall:

\begin{corspecial}
 The smallest non-zero eigenvalue of $-\Delta^{F_{\eps}}$\ is 
\begin{equation}
 \lambda_1 = 2 - 2 \eps^2 = \frac{8\pi}{\text{vol}_{\Omega}\left(\S^2\right) }.
\end{equation}
It is of multiplicity two and the eigenspace is generated by $Y_1^1$\ and $Y_1^{-1}$.
\end{corspecial}

\begin{proof}

Computation using either \eqref{eq:Delta_ylm} or directly Theorem \ref{th:laplacien_S2} gives $\Delta^{F_{\eps}} Y_1^1 = (-2 +2 \eps^2) Y_1^1$\ and $\Delta^{F_{\eps}} Y_1^{-1} = (-2 +2 \eps^2) Y_1^{-1}$. It also yields $\Delta^{F_{\eps}} Y_1^0 = -2 Y_1^0 $, now Theorem \ref{thm:spectre_S2} shows that the eigenfunctions for the first (non-zero) eigenvalue must live in the vicinity of the space generated by $Y_1^1, \; Y_1^0$\ and $Y_1^{-1}$, therefore $\lambda_1 = 2 - 2 \eps^2$.\\

Now using equations \eqref{eq:ada_sphere} we get that the Finsler volume form for $(\S^2,F_{\eps})$\ is 
$$
\Omega^{F_{\eps}} = \frac{\sinp}{\left(1-\ee\right)^{3/2}} d\theta \wedge d\phi.
$$
So
$$
\text{vol}_{\Omega} \left(\S^2\right) = \frac{4\pi}{1-\eps^2}.
$$
Hence,
\begin{equation*}
 \lambda_1 = \frac{8\pi}{\text{vol}_{\Omega}\left(\S^2\right) }.
\end{equation*}

\end{proof}

{\bf Acknowledgements.} My warmest thanks goes to both my advisors Patrick Foulon and Boris Hasselblatt for their time, constant help and our ever so fruitful discussions. Patrick suggested the construction of this operator and is therefore at the root of this work. As every work, this paper benefited hugely from discussions with many mathematicians, among them I would particularly wish to thank Bruce Boghosian,  Yves Colin de Verdi\`ere and Micka\"el Crampon. Finally, I'd like to thank Bruno Colbois for telling me about weighted Laplace operators.

\bibliographystyle{amsalpha}
\bibliography{/home/thomas/Dropbox/maths/these/tout}

\end{document}